\documentclass[review,11pt]{elsarticle}
\usepackage[margin=2.5cm]{geometry}
\usepackage{lineno,hyperref}
\usepackage{amssymb,amsmath,amsthm,latexsym,xcolor}
\usepackage{enumitem}
\usepackage{graphicx}
\usepackage{relsize}
    \usepackage{braket}
\newcommand\numberthis{\addtocounter{equation}{1}\tag{\theequation}}
\let\oldequation\equation
\let\oldendequation\endequation

\renewenvironment{equation}
  {\linenomathNonumbers\oldequation}
  {\oldendequation\endlinenomath}
\journal{~}

\newtheorem{theorem}{Theorem}[section]
\newtheorem{proposition}{Proposition}[section]
\newtheorem{lemma}{Lemma}[section]
\newtheorem{corollary}{Corollary}[section]

\numberwithin{equation}{section} 
\begin{document}

\begin{frontmatter}

\title{Distribution of values of general Euler totient function} 
\author[A]{Debika Banerjee}
\ead{debika@iiitd.ac.in}
\author[A]{Bittu Chahal}
\ead{bittui@iiitd.ac.in}

\author[A]{Sneha Chaubey}
\ead{sneha@iiitd.ac.in}

\author[A]{Khyati Khurana}
\ead{khyatii@iiitd.ac.in}
\address[A]{Department of Mathematics,\\ IIIT-Delhi,\\ New Delhi - 110020, India \vspace{0.5 cm}}
\begin{abstract}
Let $\Phi_k(n)=|\{ (x_1, x_2, \cdots, x_k)\in \left(\mathbb{Z}/n\mathbb{Z}\right)^k; \ \gcd(x_1^2+x_2^2+ \cdots+ x_k^2, n)=1\}|$ be a general totient function introduced first by 
Cald\' {e}ron et. al. 
Motivated by the classical works of Schoenberg, Erd\H{o}s, Bateman and Diamond on the distribution of $\Phi_1(n)$, we prove results on the joint distribution of $\Phi_k(n)$ for any $k\ge 1$.
Additionally, we also exhibit the extremal order of $\Phi_k(n)$. 
\end{abstract}
\end{frontmatter}
\section{Introduction}
An important function in the theory of numbers is the Euler totient function $\phi(n)$, which counts the number of integers not exceeding $n$ and relatively prime to $n$. Questions on the distribution of values of $\phi(n)$ have long been of interest, and several results exist in the literature. One of the most classical and earliest known results  in probabilistic number theory is due to Schoenberg \cite{S}, who studied the distribution of values of $\dfrac{\phi(n)}{n}$, and showed that there exists a continuous monotone function $f$ with $f(0)=0$ and $f(1)=1$ such that
\begin{align} \label{fresult}
\frac{1}{x}|\{ n \leq x; \frac{\phi(n)}{n} \leq \alpha\}| \rightarrow f(\alpha) \ \ 0\leq \alpha \leq 1; \ \mbox{as} \ x \rightarrow \infty.
\end{align}
The distribution function $f(1/\alpha)$ was investigated by Erd\H{o}s \cite{E1} with sharper error terms in \cite{W} and \cite{W1}. In particular, these authors studied the limiting function
\[\lim_{x\rightarrow\infty}\frac{1}{x}|\{n\le x: \frac{n}{\phi(n)}\ge t\}|.\] The question on distribution of $r$-th moments of $\dfrac{\phi(n)}{n}$ goes back to Chowla \cite{Ch} with improvements on the error term for special cases by Walfisz and others in \cite{Si, Wa}. Later in \cite{JL}, Liu studied the mean value estimates of the error term of $r$-th moments of $\dfrac{\phi(n)}{n}$. Further analysis related to the Riesz mean of $\left(\dfrac{\phi(n)}{n}\right)^{-1}$ was done in a series of papers \cite{SS1, SS2, SS3}, where the authors first conjectured on the error terms of $k$-th Riesz mean \cite{SS1} and later improved on their estimates of the error terms \cite{SS2, SS3}. 
Weighted means of $\left(\dfrac{\phi(n)}{n}\right)^{-1}$ was studied by Martin et.al. in \cite{MPS}, where the authors also proved the existence of the distribution function
\[\lim_{x\rightarrow\infty}\frac{1}{x}|\{n\le x: K(n)\frac{n}{\phi(n)}\le u\}|,\] for some explicit function $K(n)$. Note that the function $K(n)\dfrac{n}{\phi(n)}$ appears in the asymptotics of the average of the number of points on Elliptic curves.

Returning to \eqref{fresult}, using the existence of the distribution function $f(\alpha)$, Erd\H{o}s \cite{E} in 1945, showed that the number of integers $M(n)$ for which $\phi(m)\le n$ equals $cn+o(n).$ The constant $c=\dfrac{\zeta(2)\zeta(3) }{\zeta(6)}$ was evaluated by Dressler \cite{Dr}. Bateman derived an asymptotic formula \cite{B} given by
\[M(y)=\frac{\zeta(2)\zeta(3)}{\zeta(6)}y+O\left(y \  \text{exp}(-(1-\epsilon)(1/2\log y\log\log y)^{1/2})\right).\] 
The function $M(y)$ counts the number of points $(n,\phi(n))$ lying in the semi-infinite horizontal strip $\{(s,t) : 0<s<\infty, 0< t\le y\}.$ Later, Diamond \cite{D} looked at the question for rectangles. He studied the function $\psi(x,y)$ counting the number of points $(n,\phi(n))$ lying in rectangle $(0,x]\times(0,y]$ and proved that
\[\psi(x,y)=xg(y/x)+O\left(y\exp(-(c\log ey\log\log e^2y)^{1/2})\right),\]
where $g$ is a continuous increasing function on $(0,1)$ determined as a contour integral. 


\par 
Our aim in this paper is to study an analog distribution question as above for $\Phi_k(n)$. The function $\Phi_k(n)$ is the cardinality of the set of $k$-tuples \[S:=\{ (x_1, x_2, \cdots, x_k)\in \left(\mathbb{Z}/n\mathbb{Z}\right)^k; \ \gcd(x_1^2+x_2^2+ \cdots+ x_k^2, n)=1\}.\]
Sometimes, in investigating solutions of a Diophantine equation, it is important to study the set of non-solutions of the equation. The set $S$ is a subset of the set of non-solutions of the 
 Diophantine equation $x_1^2+x_2^2+ \cdots+ x_k^2=n$ and was first introduced in \cite{CG}.
The function 
\begin{align}\label{Genphi}
\Phi_k(n)=|\{ (x_1, x_2, \cdots, x_k)\in \left(\mathbb{Z}/n\mathbb{Z}\right)^k; \ \gcd(x_1^2+x_2^2+ \cdots+ x_k^2, n)=1\}|.
\end{align}
can also be viewed as a 
generalization of the Euler totient function. 
Analogs and generalizations of the Euler totient function have been studied before; see, for example, \cite{Kac} and \cite{Zh}, and the references in \cite{CG}. A famous generalization is Jordan's totient function, named after Camille Jordan. It is given by
\begin{align*}
J_k(n)&:=|\{ (a_1, a_2, \cdots, a_k); \ 1 \leq a_i  \leq n; \ \gcd(a_1, a_2, \cdots, a_k, n) = 1\}|\\&=n^{k} \prod_{p|n} \left(1-\frac{1}{p^k}\right).
\end{align*}
It is clear that $\Phi_1(n)=J_1(n)=\phi(n)$. Furthermore, for
$k = 2, 4$ and $8$, $\Phi_k(n)$ coincides with the number of units in the rings of Gaussian integers, quaternions, and octonions over $\mathbb{Z}/n\mathbb{Z}$ respectively.
 More precisely, $\Phi_2(n)$, denoted 
by GIphi$(n)$ \cite{C} in the literature, represents the number of Gaussian integers $\mathbb{Z}[i]$ in a reduced
residue system modulo $n$. It was shown in \cite{CG} that the function $\Phi_k(n)$ is multiplicative for every $k$ and has the following explicit formula in terms of the prime-power decomposition of $n$:
For every $n \in \mathbb{N}$,
\begin{align}
\Phi_k(n)=\begin{cases}
n^{k-1} \phi(n)&\quad   \text{if } k \text{ is odd}, \\
n^{k-1} \phi(n) \prod_{\substack{p|n\\p>2}}\  \left(1-\frac{(-1)^{\frac{k(p-1)}{4}}}{p^{\frac{k}{2}}}\right)& \quad \text{if } k \text{ is even}. 
\end{cases}
\end{align}
Moreover, $\Phi_k(n)$ can be rewritten as 
\begin{align}
\Phi_k(n)=\begin{cases}
n^{k-1} \phi(n)&  \quad \text{if } k\equiv 1 \pmod 2, \\
n^{k-1} \phi(n) \prod_{\substack{p|n\\p>2}}\  \left(1-\frac{1}{p^{\frac{k}{2}}}\right)& \quad \text{if  } k\equiv 0 \pmod 4 ,\\
 n^{k-1} \phi(n) \prod_{\substack{p|n\\p \equiv 1  \pmod  4}}\  \left(1-\frac{1}{p^{\frac{k}{2}}}\right) \prod_{\substack{p|n\\p \equiv -1  \pmod  4}}\  \left(1+\frac{1}{p^{\frac{k}{2}}}\right)& \quad \text{if } k\equiv 2\pmod 4.\\
\end{cases}
\end{align}
The authors also derived an asymptotic formula for the summatory function $\sum_{n \leq x} \Phi_k(n)$. 
\subsection{Main Results.}
For any real number $\beta$, we define the quantity
\begin{align}\label{Function}
\phi_{\Phi_k, \beta}(x, y)=|\{ n \leq x; \ \  \frac{\Phi_k(n)}{n^{\beta} }\leq y\}|.
\end{align}
Note that $\frac{1}{x}\phi_{\Phi_k, \beta}(x, y)$ is the ratio of integers $n\le x$ for which $(n,\Phi_k(n))$ lies below the curve $t=s^{\beta}y$ in the $s$-$t$ plane.
The above quantity can be rewritten as
\begin{align}\label{Function1}
\phi_{\Phi_k, \beta}(x, y)
= \begin{cases} 
|\{ n \leq x; \ \  \frac{\phi(n)}{n^{\delta} }\leq y\}| &\quad  \text{if } k \text{ is odd}, \\
|\{ n \leq x; \ \  \frac{\phi(n) \alpha_{k}(n) }{n^{\delta} }\leq y\}|&  \quad  \text{if } k  \text{ is even}. 
\end{cases}
\end{align}
where $\delta= \beta-(k-1)$, and $\alpha_{k}(n)$ is a multiplicative function given by
\begin{align} \label{olx}
\alpha_{k}( p^m)=\begin{cases}
1&\quad   \text{if }  p=2,\\
\left(1-\frac{(-1)^{\frac{k(p-1)}{4}}}{p^{\frac{k}{2}}}\right)&  \quad \text{otherwise.}
\end{cases}
\end{align}
for any $m \geq 1$. 
Let
\begin{align} \label{Rvalue}
  R_{k, \beta}(z)=\begin{cases}
   \prod_{p}\left(1-\frac{1}{p}+\frac{1}{p}\left(\frac{p}{p-1}\right)^z \right)&  \quad \text{if } k \text{ is odd},\\
   \prod_{p}\left(1-\frac{1}{p}+\frac{1}{p}\left(\frac{p}{p-1}\right)^z \alpha_k(p)^{-z}\right)&  \quad \text{if } k \text{ is even}.
  \end{cases}
 \end{align}
We use the notation $\alpha=\frac{y}{x}$ for the rest of the paper. The notation $x_{{\beta}, k}$ would mean a positive real number which may not be the same at each occurrence.
Similarly, $\varepsilon>0$ will stand for an arbitrarily small positive real number which may not be the same at each occurrence. 
We will use the value $\delta=\beta-(k-1)$. Our first main result is the following theorem:
\begin{theorem}\label{thm3}
For $0< \delta <1$ with $y= \frac{(\log( (1-\delta)^{-1}+\varepsilon))^{-4}}{c_k \kappa^4}x^{1-\delta}$, for any arbitrarily small $\varepsilon>0$ and $R_{k, \beta}(z)$ as defined in \eqref{Rvalue}, there exists a positive constant $p_k$ such that
\begin{align*}
\phi_{\Phi_k, \beta}(x, y)=R_{k, \beta} \left(\frac{1}{(1-\delta)}\right)y^{\frac{1}{(1-\delta)}}+O_\delta\left( y^{\frac{1}{(1-\delta)}}\exp\left( -\sqrt{p_k     \log y\log\log y}\right)\right),     
\end{align*}
for all $x\geq x_{{\beta}, k}$.
\end{theorem}
For other values of $y$, we obtain the following estimate.
\begin{theorem}\label{thm4}
For $0< \delta <1$ with $0< y \leq x$ and $ y \neq \frac{(\log ((1-\delta)^{-1}+\varepsilon))^{-4}}{c_k \kappa^4}x^{1-\delta}$, for any arbitrarily small $\varepsilon>0$, there exists a positive constant $p_k^{\prime}$ such that
\begin{align*}
\phi_{\Phi_k, \beta}(x, y)\ll_\delta x^{1-(1-\delta)\varepsilon}y^{\varepsilon} \exp\left( -\sqrt{p_k^{\prime}     \log x\log\log x}\right),     
\end{align*}
for all $x\geq x_{{\beta}, k}$.
\end{theorem}
Note that for $0<\delta<1,$ if $y\ge x$, then the value of $\phi_{\Phi_k, \beta}(x, y)=\lfloor x\rfloor$.
\begin{theorem}\label{thm5}
For $\delta <0$ and $0<\alpha <1$, for any arbitrarily small $\varepsilon>0$, there exists a positive constant $q_k$ such that
\begin{align*}
\phi_{\Phi_k, \beta}(x, y)\ll_{\delta} x^{\frac{1}{2}+(1-\delta)\varepsilon} y^{\frac{1}{2(1-\delta)}-\varepsilon} \exp\left( -\sqrt{q_k\log x\log\log x}\right),
\end{align*}
for  $x\geq x_{{\beta}, k}$.
\end{theorem}
In particular, for $\delta=0$, i.e., $\beta=k-1$, we obtain
\begin{theorem}\label{thm6}
For  $0<\alpha <1$, there exists a positive constant $r_k$ such that
\begin{align*}
\phi_{\Phi_k, k-1}(x, y)= R_{k, k-1}(1)y+ O\left(y\exp\left( -\sqrt{r_k\log y\log\log y}\right)\right),
\end{align*}
for  $x\geq x_{{\beta}, k}$.
\end{theorem}
On substituting $k=1$ in the above theorem, we recover Diamond's \cite{D} result.
\begin{corollary}\label{cor1}
For $\beta=0$, and $0<\alpha<1$, we have
\[\phi_{\Phi_1, 0}(x, y)=\frac{\zeta(2)\zeta(3)}{\zeta(6)}y+O\left(y\exp\left( -\sqrt{c\log y\log\log y}\right)\right).\]
\end{corollary}
The only case left to consider is when $\delta<0$ and $\alpha\ge 1$.
\begin{theorem} \label{thm7}
For $\delta <0$ and $\alpha\ge 1$, we have
\begin{align*}
\phi_{\Phi_k, \beta}(x, y)= R_{k, \beta} \left(\frac{1}{(1-\delta)}\right)y^{\frac{1}{(1-\delta)}}   +O_\delta( x \alpha^{\frac{1}{\log x}}     (\log x \log\log x )^{\frac{1}{2}} ),   
\end{align*}
whenever $x (\log x \log\log x )^{\frac{1}{2}} < e^{1-\delta}y^{\frac{1}{1-\delta}-\frac{1}{\log x}}$, otherwise
\begin{align*}
\phi_{\Phi_k, \  \beta}(x, y)      \ll_{\delta}  x
\end{align*}
for  all $x\geq x_{{\beta}, k}$.
\end{theorem}
\subsection{Remarks on Proofs.}
The proofs primarily involve three steps. First, 
we study a two-variable analytic function $F_{k,\beta}(s,z)=\sum_{n\ge 1}n^{-s+\beta z}{\Phi_k(n)}^{-z}$, compute its Euler product representation which differs for $k$ odd or even, and obtain bounds for $\frac{F_{k,\beta}(s,z)}{\zeta(s+z-\delta z)}$. The next step is essentially the key step devoted to establishing an asymptotic formula for the average of the coefficient of  $F_{k,\beta}(s,z)$. A careful application of effective Perron's formula and bounds on the zeta function, along with estimates obtained in Step 1, help to achieve this. This is also where the method of proof differs from the method used by Diamond \cite{D} for his case of the usual Euler totient function. Diamond used the following integral identity
 \[ \int_1^{\alpha x}\Phi(x,u)\frac{du}{u}=\frac{x}{2\pi i} \int_\Gamma \prod(1-z,z)\frac{\alpha^{z} }{z^2(1-z)}\,dz+ O(x\ \exp( -\sqrt{c^\prime\log x\log\log x})), \numberthis\label{integration}\]
thereafter using the method of differencing arrived at his distribution result. 
We do not establish analogous relation \eqref{integration} for the general totient function. Instead, using our asymptotic formula and the summation formula \[|h(y)-\frac{1}{2\pi i}\int_{a-i T}^{a+i T^\prime}\frac{y^z}{z} dz|\ll\frac{y^a}{|\log y|}\left(\frac{1}{T}+\frac{1}{T^\prime}\right)  \quad(y\ne1), \numberthis\label{result}\] where 
$h(y)=1, 1/2$ or $0$ depending on $y> 1, y=1,$ and $0<y<1$, respectively, 
  we arrive at our results. The above identity enables us to avoid using differencing. In the final step, we examine the error terms obtained in Step 2 for different intervals of $\delta$ and $\alpha$, and an application of Perron's (Proposition \ref{Perron0}) yields the main theorems.
  \subsection{Results on the extremal order.}
We end this article by stating results on the maximal and minimal orders of $\Phi_k(n)$. 
For an arithmetic function $f$ and a non-decreasing function $g$ which is positive for all $x\ge x_0$ for some $x_0$, we say that $g$ is a maximal order (resp. a minimal order) for $f$ if we have
\[\lim\sup_{n\rightarrow\infty}\frac{f(n)}{g(n)}=1\quad \left(\text{resp. }\lim\inf_{n\rightarrow\infty}\frac{f(n)}{g(n)}=1\right).\]
Note that the maximal order of $\phi(n) $ is $n$ and the minimal order is $\dfrac{e^{-\gamma} n}{\log\log n}$ \cite[page 115]{T}.
\begin{theorem}\label{kodd}
For any $k\ge 1$, the maximal order of $\Phi_k(n) $ is
$n^k$. 
\end{theorem}
The minimal order of $\Phi_k(n)$ depends on the parity of $k$ and is given by
\begin{theorem}\label{keven}The minimal order of $\Phi_k(n)$ is
\[ 
\begin{cases}
    \ \frac{e^{-\gamma}n^k}{\log\log n}, &   \text{if } k\equiv 1\pmod 2, \\
    \left( \frac{2^{\frac{k}{2}}     }{2^{\frac{k}{2}}-1}\right)\zeta\left(\frac{k}{2}\right)^{-1}  \frac{ e^{-\gamma}n^k}{\log\log n} &         \text{if } k\equiv 0\pmod 4,\\
    \  L\left(\frac{k}{2}, \chi_1\right)^{-1} \frac{e^{-\gamma}n^k}{\log\log n} &           \text{if } k\equiv 2\pmod 4.
\end{cases}
\]
\end{theorem}
\section{Organization} In Section \ref{Prelim}, we review some critical results on primes and summation formulae that will be needed in the proofs of key lemmas illustrated in section \ref{key} and extremal order in section \ref{extremal}. In subsection \ref{generate}, we describe the generating function. In subsection \ref{boundssec}, we prove upper bounds for the functions appearing in the generating function computed in subsection \ref{generate}. 
In Section \ref{key}, we derive asymptotics of the summatory functions, which are also the key results that will be needed in the proofs of our main theorems. In Section \ref{Mainthm}, we establish the proofs of our main theorems. Finally, in section \ref{extremal}, we deduce proofs of Theorems \ref{kodd} and \ref{keven}. 
\section{Funding}
This work was supported by the University Grants Commission, Department of Higher Education, Government of India [191620135578 to B.C., 191620205105 to K.K.]; and the Science and Engineering Research Board, Department of Science and Technology, Government of India [SB/S2/RJN-053/2018 to S.C.]. The authors are grateful to the referee for valuable suggestions in an earlier version of the paper.
\section{Preliminaries}\label{Prelim}
This section reviews several crucial results needed to derive our main theorems. Some results can be obtained from earlier works, while some require proof.
We first summarize Merten's results on the asymptotic behavior of quantities depending on primes.
\begin{proposition}\label{prime}
For all $x \geq 2$,
\begin{align*}
\sum_{p\leq x} \frac{\log p}{p} = \log x+R(x),
\end{align*}
where $R(x)=O(1)$.
Furthermore, 
\begin{align} \label{Merten1}
\sum_{p\leq x} \frac{1}{p} = \log \log x+b_0+O\left(\frac{1}{\log x}\right),
\end{align}
where $b_0=1-\log \log 2+\int_{2}^{\infty} \frac{R(u)}{u (\log u)^2} du$.
\end{proposition}
\begin{proof}
For proof, see \cite[page 50]{MV}.
\end{proof}
An immediate corollary to Proposition \ref{Merten1} yields Merten's formula
\begin{proposition}\label{primemerten}
 For $x \geq 2$,
\begin{align*}
\prod_{p\leq x} \left(1-\frac{1}{p}\right) =\frac{e^{-\gamma}}{\log x} \left(1+O\left(\frac{1}{\log x}\right) \right),
\end{align*}
where $\gamma$ is the Euler's constant.
\end{proposition}
\begin{proof}
See \cite[page 19]{T}.
\end{proof}
\begin{proposition}\label{mertengeneral}For $x \geq 2$ and $j \geq 2$,
\begin{align*}
\prod_{p\leq x} \left(1-\frac{1}{p^j}\right) =\zeta(j)^{-1} +O\left(\frac{1}{x^{j-1}} \right).
\end{align*}
\end{proposition}
\begin{proof}
The proof follows from the identity
\begin{align*}
\log \zeta(j)= -\sum_{p\leq x} \log \left(1-\frac{1}{p^j}\right)+O\left(\int_{x}^{\infty} \frac{1}{t^{j}} dt\right).
\end{align*}
\end{proof}
\begin{proposition}\label{mertengeneralap}Let $\chi_1$ be the non-principal Dirichlet character mod $4$. For $x \geq 2$ and $j \geq 1$,
\begin{align*}
\prod_{ p\leq x} \left(1-\frac{\chi_1(p)}{p^j}\right) =L(j, \chi_1)^{-1} +\begin{cases}
    & O\left(\frac{1}{x^{j-1}} \right) \ \ \mbox{if} \ j>1,\\
    & O\left(\frac{1}{\log x} \right)\ \ \mbox{if} \ j=1.
\end{cases}
\end{align*}
\end{proposition}
\begin{proof}
The proof is trivial.
\end{proof}

Next, we summarize the summation formulae.

\begin{proposition}\label{P1}
For any positive $a, T, T^{\prime}$, we have
\begin{align*}
\frac{1}{2\pi i}\int_{a-iT}^{a+iT^{\prime}}\frac{y^s}{s} ds=h(y)+ O\left(\frac{y^a}{|\log y|} \left(\frac{1}{T}+\frac{1}{T^{\prime}} \right)\right) \quad(y\neq 1),
\end{align*}
where $h(y)$ is defined as the following:
\begin{align*}
h(y)=\begin{cases}
1& \quad y>1,\\
\frac{1}{2}& \quad y=1,\\
0& \quad 0<y<1.
\end{cases}
\end{align*}
\end{proposition}

\begin{proof}
For proof, see \cite[page 218]{T}.
\end{proof}
As a direct application of Proposition \ref{P1}, we obtain Perron's formula. 
Let $f(n)$ be an arithmetic function and $F(s)$ be its Dirichlet series 
with abscissa of absolute convergence $\sigma_a$. Let $A(x)$ be the normalized summatory function 
\begin{align*}\label{sum}
A(x)=\sum^{\prime}_{n \leq x} f(n),
\end{align*}
where $\sum^{\prime}$ signifies that the coefficient of $f(n)$ is $\frac{1}{2}$ when $x \in \mathbb{N}$.
\begin{proposition}[Perron's formula]\label{Perron}
For $a> \max(0, \sigma_a)$, $T\geq 1$ and $x \geq 1$, we have
\begin{align*}
A(x)=\frac{1}{2\pi i} \int_{a-iT}^{a+iT} F(s) \ x^s\frac{ds}{s} +O\left(\frac{x^{a}}{T}\sum_{n\geq1} \frac{|f(n)|}{n^a |\log (x/n)|} \right).
\end{align*}
\end{proposition}

\begin{proof}
The proof follows from Proposition \ref{P1}. 
\end{proof}

To handle the counting function in $\eqref{Function}$, we introduce a general counting function 
for any real-valued arithmetic function $f$ and $x, y\ge 1, z \in \mathbb{C}$, 
\begin{align*}
\phi_{f}(x, y):=\sum_{\substack{n \leq x\\ f(n) \leq y}}1=|\{ n \leq x; \ \  f(n) \leq y\}|,
\end{align*}
In the preceding lemma, we estimate $\phi_{f}(x, y)$ using the summatory function
\[A_z(x)=\sum_{n \leq x} f(n)^{-z}.\numberthis\label{Az}\]

\begin{proposition}\label{Perron0}
Let $f$ be any real-valued arithmetical function. Also, suppose $f(n) \geq 0$ for every $n \geq 1$. Then for any $U, U^{\prime} \geq 1$ and $b>0$, we have
\begin{align*}
\phi_{f}(x, y)=\frac{1}{2\pi i}\int_{b-iU}^{b+iU^{\prime}} A_z(x) \frac{y^z}{z} dz+O\left (y^{b} \sum_{n \leq x} f(n)^{-b} \left(\frac{1}{U}+\frac{1}{U^{\prime}} \right)\right).
\end{align*}
where $A_z(x)$ as defined in \eqref{Az}.
\end{proposition}

\begin{proof}
The proof follows directly from Proposition \ref{P1} and the fact that 

\begin{align*}
\int_{b-iU}^{b+iU^{\prime}} \left( \sum_{n \leq x} f(n)^{-z}  \right) \frac{y^z}{z} dz= \sum_{n \leq x} \int_{b-iU}^{b+iU^{\prime}} f(n)^{-z}   \frac{y^z}{z} dz.
\end{align*}

\end{proof} 
In our applications of Perron's formula, we shall use the following estimates for the Riemann zeta function. For large $|t|\geq 1$,
\begin{align*}
\zeta(\sigma+it)\ll \begin{cases}
 |t|^{\varepsilon}& \quad \sigma \geq 1,\\
|t|^{\frac{1}{2}(1-\sigma)} &\quad 0<\sigma < 1,\\
|t|^{\frac{1}{2}-\sigma} &\quad \sigma\leq 0.
\end{cases} \numberthis \label{zetabounds}
\end{align*}
Finally, we derive bounds on the coefficients $\alpha_{k}(p)$. 
From \eqref{olx}, we see that for $k\equiv 2\pmod 4$ and $p\equiv 3\pmod 4$, $\alpha_{k}(p)\le\left(1+2^{-k/2}\right)$ and $\alpha_{k}(p)\le 1$, otherwise. Similarly, one gets the lower bound 
\begin{align*}
\alpha_{k}(p) \geq \begin{cases}
1& \quad \text{if } k\equiv 2\pmod 4 \text{ and } p \equiv 3 \pmod{4},\\
 \left(1-\frac{1}{2^{k/2}} \right)&  \quad \text{otherwise}. 
\end{cases}
\end{align*}
Combining these, we obtain
\begin{proposition}\label{bound}
For any prime $p$ and the arithmetical function $\alpha_k$ defined in \eqref{olx}
\begin{align*} 
\left(1-\frac{1}{2^{k/2}} \right) \leq \alpha_{k}(p) \leq  \left(1+\frac{1}{2^{k/2}} \right).
\end{align*}

\end{proposition}

\medskip
\section{First steps} \label{First}
\subsection{Generating Function.} \label{generate}
Let $s=\sigma+it$ and $z=\xi+i\eta$ be complex variables. Let $S_\alpha =\{(s,z):\sigma+(1-\delta)\xi >\alpha\}$, define
\[F_{k, \beta}(s,z)=\sum_{n=1}^{\infty}n^{-s+\beta z}{\Phi_k(n)^{-z}}.\]
The series converges on $S_{1}$ since $\Phi_k(n)\gg n^k/(\log\log n)^2$. Since $\Phi_k(n)^{-z}$ is multiplicative so $F_{k, \beta}$ has an Euler product representation on $S_{1}$ which depends on the parity of $k$. 
\begin{description}
\item[Case (1)] For odd $k$, we have the Euler product representation of $F_{k, \beta}$ on $S_{1}$ as 
\begin{align}\label{G}
    F_{k, \beta}(s,z)&=\prod_{p}\{1+p^{-s+\delta z} \phi(p)^{-z}+p^{-2s+2\delta z}\phi(p^2)^{-z}+p^{-3s+3\delta z}\phi(p^3)^{-z}+\cdots \} \nonumber \\
    &=\prod_p\{1+p^{-s+\delta z}(p-1)^{-z}(1+p^{-s+\delta z-z}+p^{-2s+2\delta z-2z}+\cdots) \}       \nonumber\\
     &=\prod_p\{1+p^{-s+\delta z}(p-1)^{-z}(1-p^{-s+\delta z-z})^{-1} \}       \nonumber\\
     &=\prod_p\{1-p^{-s+\delta z-z}+p^{-s+\delta z}(p-1)^{-z} \}\zeta(s+z-\delta z)  \nonumber \\
     &=G_{k, \beta}(s, z)\zeta(s+z-\delta z),
\end{align}
where $\zeta$ is the Riemann zeta function. \item[Case (2)] When $k$ is even, then the Euler product representation of $F_{k, \beta}$ on $S_{1}$ is computed to be
\begin{align}\label{I}
  F_{k, \beta}(s,z)&=\prod_{p}\{1+\alpha_k(p)^{-z}(p^{-s+\delta z} \phi(p)^{-z}+p^{-2s+2\delta z} \phi(p^2)^{-z}+p^{-3s+3\delta z} \phi(p^3)^{-z}+\cdots) \}  \nonumber \\
    &=\prod_p\{1+p^{-s+\delta z}(p-1)^{-z} \alpha_k(p)^{-z}(1+p^{-s+\delta z-z}+p^{-2s+2\delta z-2z}+\cdots) \}  \nonumber\\
     &=\prod_p\{1+p^{-s+\delta z}(p-1)^{-z} \alpha_k(p)^{-z}(1-p^{-s+\delta z-z})^{-1} \}  \nonumber\\
     &=\prod_p\{1-p^{-s+\delta z-z}+p^{-s+\delta z}(p-1)^{-z}\alpha_{k}(p)^{-z} \}\zeta(s+z-\delta z)  \nonumber \\
     &=I_{k, \beta}(s, z)\zeta(s+z-\delta z).
 \end{align}
 \end{description}
\subsection{Upper Bounds for $I_{k, \beta}(s, z)$ and $G_{k, \beta}(s, z)$} \label{boundssec}
It is naturally desirable as well as required in our proofs to have explicit upper bounds for $I_{k, \beta}(s, z)$ and $G_{k, \beta}(s, z)$ in a region inspired by the holomorphicity of $F_{k,\beta}(s,z)$. The region we consider is: for $r\ge 0$, let
\begin{align*}
S_{r}^+=\{ (\sigma+it, \xi+i\eta) ; \      \sigma\geq 0, \    (1-\delta)\xi \geq 0;\         \sigma +(1-\delta)\xi >r\}.
\end{align*}


\begin{lemma}\label{bound12}
If $k$ is an even number, then the product defining $I_{k, \beta}(s, z)$ converges and defines an analytic function of $s$ and $z$ on $S_{0}^+$. Moreover,
\begin{align*}
I_{k, \beta}(s, z)  \ll \left\{\begin{array}{cc}
  \exp\left(\dfrac{M_k\log |\eta|}{\log\log |\eta|}\right)  &\quad \text{if }\ \sigma + (1-\delta)\xi \ge 1-\log\log |\eta|/\log |\eta| \text{ and }\ \log\log |\eta|\ge 10, \\ 
     1 & \quad \text{if }\ \sigma +(1-\delta)\xi \ge \frac{1}{2} \text{ and }\ |\eta|<\exp\exp{10}.
    \end{array}\right.\end{align*}
    \end{lemma}

The estimates are valid independent of $t$.

\begin{proof}


Using the lower bound for $\alpha_{k}(p)$ in Proposition \eqref{bound}, for $\xi>0$,  one trivially has
\begin{align}\label{eqn:2.0}
     |\alpha_{k}(p)^{-z}(p-1)^{-z}-p^{-z}| \ll_k (p-1)^{-\xi}. 
\end{align}
The lower bound for $\alpha_{k}(p)$ in Proposition \eqref{bound} yields the estimate
\begin{align}\label{eqn:2.1}
     |\alpha_{k}(p)^{-z}(p-1)^{-z}-p^{-z}| =&\left|z\int_{\alpha_{k}(p)(p-1)}^p{u^{-z-1}}\,du\right| \ll_k  |z|(p-1)^{-\xi-1}.          \end{align} 
For  $\log \log |\eta|\geq10$ and  $\sigma +(1-\delta)\xi \ge 1- \varepsilon $, we split the product defined by $I_{k, \beta}(s, z)$ into two parts
 \begin{align}\label{eq2}
 |  I_{k, \beta}(s, z) |= |\prod_{p\leq |\eta|}|.|\prod_{p>|\eta|}|.
 \end{align} 
 The finite part can be estimated as follows:
         \begin{align}\label{eq3}
      |\prod_{p\leq |\eta|}| 
      &=|\prod_{p\leq \eta}\{1+p^{-s+\delta z}( {( p-1)^{-z}\alpha_{k}(p)^{-z}} -{p^{-z}})\}| \leq\prod_{p\leq\eta}\{1+ A_k(p-1)^{-\sigma-(1-\delta)\xi}\} \nonumber \\
      &\leq\exp\left(\sum_{p\leq\eta}\{A_k(p-1)^{-\sigma-(1-\delta)\xi}\}\right) \nonumber \leq\exp\left(A_k+A_k\int_{3/2}^{\eta}v^{-\sigma-(1-\delta)\xi}d\pi (v)\right)  \nonumber \\
  &\leq\exp\left(A_k+A_k\int_{3/2}^{\eta}v^{-1+\varepsilon}\frac{dv}{\log v}\right), 
    \end{align}
    where $A_k=\frac{2^{k+2}}{(2^{k/2}-1)^2}.$
   Setting $w:=\varepsilon\log v$ into \eqref{eq3}, it follows that
    \begin{align}\label{Mi}
        |\prod_{p\leq |\eta|}|  &\leq  \exp\left(A_k+A_k\int_{\varepsilon \log(\frac{3}{2})}^{\varepsilon \log(|\eta|)}\frac{e^w}{w}\,dw\right)        \nonumber \\
         &= \exp\left(A_k+A_k\left(\int_{\varepsilon \log(\frac{3}{2})}^{1} \frac{e^w}{w}\,dw+ \int_1^{\frac{\varepsilon \log |\eta|}{2}}  \frac{e^w}{w}\,dw+       \int_{ \frac{\varepsilon \log |\eta|}{2}}^{\varepsilon \log|\eta|} \frac{e^w}{w}\,dw\right)\right) \nonumber \\
         &=\exp(A_k+A_k(I_1+I_2+I_3)).
        \end{align}
        The first integral is estimated as
 \begin{align}\label{eq11}
   I_1=\int_{\varepsilon \log(\frac{3}{2})}^{1} \frac{e^w}{w}\,dw
   \leq\frac{e}{\varepsilon \log\frac{3}{2}}\int_{\varepsilon \log(\frac{3}{2})}^{1}\,dw
   \leq\frac{e}{\varepsilon \log\frac{3}{2}}
   \ll\frac{1}{\varepsilon}.  \end{align}
Next, simple use of integration by parts yields an estimate for the second integral in \eqref{Mi} 
 \begin{align}\label{eq12}
     I_2=\int_1^{ (1/2)\varepsilon \log |\eta| }  \frac{e^w}{w}\,dw
   \leq e^{(1/2)\varepsilon \log |\eta
    |}.\end{align}
For the third integral, we have
\begin{align}\label{eq13} 
    |I_3|=\int_{(1/2)\varepsilon \log |\eta|}^{\varepsilon \log |\eta|}\frac{e^w}{w}\ dw  \leq \frac{e^{\varepsilon \log |\eta |}}{\varepsilon \log |\eta|}+e^{(1/2)\varepsilon \log |\eta |}\frac{1}{\varepsilon \log |\eta|}
    \ll \frac{e^{\varepsilon \log |\eta |}}{\varepsilon \log |\eta|}.\end{align}
Selecting $\varepsilon=\log\log |\eta|/\log |\eta|$ optimally from \eqref{eq11}, \eqref{eq12} and \eqref{eq13}, and using it in \eqref{Mi}, we obtain
 \[ |\prod_{p\leq |\eta|}|\ll \exp\left(\frac{A_k\log|\eta|}{\log\log|\eta|}\right).   \]
Next, we consider the infinite part in \eqref{eq2}, with $B_k=2\left(\frac{2^{k/2}}{2^{k/2}-1} \right)^3$
 \begin{align*}
      |\prod_{p> |\eta|}| 
      &\leq |\prod_{p> |\eta|}\{1+p^{-s+\delta z}  (  {\left( p-1\right)^z \alpha_{k}(p)^{-z}} -{p^{-z}}) \} |\leq \prod_{p> |\eta|}\{1+B_k|z|p^{-\sigma+\delta \xi}(p-1)^{-\xi-1} \}\\
       &\leq \prod_{p> |\eta|}\{1+C_k|z|(p-1)^{-\sigma-(1-\delta)\xi-1}\}\leq \exp\left(\sum_{p>|\eta|} C_k|z|(p-1)^{-\sigma-(1-\delta)\xi-1} \right)\\
       &\leq \exp \left(C_k|z|(|\eta|-1)^{-2+\varepsilon}+C_k|z|\int_{|\eta|}^{\infty}v^{-2+\varepsilon}d\pi (v)\right  )\leq \exp\left(D_k|z|\int_{|\eta|}^{\infty}v^{-2+\varepsilon}\frac{dv}{\log v} \right  )\\
      &\leq \exp\left( \frac{D_k|z||\eta|^{\varepsilon-1}}{(1-\varepsilon)\log |\eta|}  \right)=O(1),
 \end{align*}
where $C_k=4B_k$, and $D_k=2C_k$. For  $\log \log |\eta|<10$ and  $\sigma +(1-\delta)\xi \ge 1/2 $, we have
 \[|  I_{k, \beta}(s, z) |\leq\exp\left(\sum_{p} C_k|z|(p-1)^{-\sigma-(1-\delta)\xi-1} \right)=O(1).\]
 For any $\varepsilon>0,\ M<\infty$, on each set $S_\varepsilon^+\cap \{z:|z|\leq M\}$ the product converges uniformly since
 \begin{align*}
     &\left|\log \left(\prod_{a<p<b}\{1+p^{-s+\delta z}( {( p-1)^{-z}\alpha_k(p)^{-z}} -{p^{-z}})\}\right)\right| \ll\sum_{a<p<b}\left(\frac{1}{p^{\sigma-\delta \xi}} \left|\alpha_k(p)^{-z} (p-1)^{-z} - p^{-z}  \right|\right)\\
     &\ll\sum_{a<p<b}  \left(\frac{B_k|\xi|}{(p-1)^{\xi+\sigma-\delta \xi+1}}  \right)\rightarrow 0
     \end{align*}
    uniformly as $a,b\rightarrow \infty$ and $(s,z) \in S_\varepsilon^+\cap \{z:|z|\leq M\}$. Thus the product defines an analytic function of $s$ and $z$ on $S_\varepsilon^+$.
 
 \end{proof}
Similarly, one can deduce upper bounds for the function $G_{k, \beta}(s, z)$, and we omit the proof here.
 \begin{lemma}\label{bound11}
If $k$ is an odd number, then the product defining $G_{k, \beta}(s, z)$ converges and defines an analytic function of $s$ and $z$ on $S_{0}^+$. Moreover
\begin{align*}
 G_{k, \beta}(s, z) \ll \left\{\begin{array}{cc}
  \exp\left(\dfrac{3\log |\eta|}{\log\log |\eta|}\right) & \quad \text{if }\sigma +(1-\delta)\xi \ge 1- \frac{\log\log |\eta|}{\log |\eta| } \text{ and }\log\log |\eta|\ge 10, \\
     1   & \quad \text{if }\ \sigma +(1-\delta)\xi \ge \frac{1}{2}\text{ and }|\eta|<\exp\exp 10.
    \end{array}\right.\end{align*}
    \end{lemma}
Next, we deal with the particular value of $s=1-(1-\delta)z$, in which case we obtain the following. 
\begin{lemma}\label{bound13}
For $\xi>0$ and $(s, z) \in  S_{1-\varepsilon}^+$, there exists a constant $\kappa$ such that 
\begin{align*}
I_{k, \beta}(1-(1-\delta)z, z))\ll  \exp \left( \xi \left( \log c_k+ 4 \log \log \xi+4\log \frac{\kappa}{e} \right)\right),
\end{align*}
with $c_k=(1-\frac{1}{2^{k/2}} )^{-1}$. 
\end{lemma}
\begin{proof}
As before, we split the product $I_{k, \beta}(s, z)$ in two parts $ | I_{k, \beta}(s, z) |= |\prod_{p\leq \xi}|.|\prod_{p>\xi}|$. The first part can be estimated as the following:
\begin{align*}
|\prod_{p\leq \xi}| & \ll  \prod_{p\leq \xi}\{1-\frac{1}{p}+\frac{1}{p}\left(\frac{p}{p-1}\right)^{\xi} \alpha_k(p)^{-\xi}          \}\ll  \prod_{p\leq \xi} \left(\frac{p}{p-1}\right)^{\xi}  \alpha_k(p)^{-\xi}   \left\{\left(1-\frac{1}{p}\right) \alpha_k(p)^{\xi}   +\frac{1}{p}        \right\}\\
& \ll  \prod_{p\leq \xi} \left(\frac{p}{p-1}\right)^{\xi}  \alpha_k(p)^{-\xi}   \left\{\left(1-\frac{1}{p}\right) 2^{\xi}   +\frac{ 2^{\xi}  }{p}        \right\}\ll  c_k^{\xi} \prod_{p\leq \xi} \left(\frac{2p}{p-1}\right)^{\xi},
\end{align*}
where in the last step, we employed bounds for $\alpha_k(p)$ in Proposition \ref{bound}.  
Using Merten's estimate and \eqref{Merten1} from Proposition \ref{prime} with $b_0=\log \frac{\kappa}{e}$, we have
\begin{align}\label{bound1}
|\prod_{p\leq \xi}| &\ll  c_k^{\xi} \exp \left(\sum_{p\leq \xi} \frac{2\xi}{(p-1)}\right)       \nonumber \ll  \exp \left( \xi \left( \log c_k+ 4\sum_{p\leq \xi} \frac{1}{p} \right)\right)         \nonumber \\
&\ll  \exp \left( \xi \left( \log c_k+ 4 \log \log \xi+4\log \frac{\kappa}{e} \right)\right).
\end{align}
For the infinite part, we have 
\begin{align}\label{bound2}
|\prod_{p> \xi}| &\ll  \prod_{p> \xi}\left(1-\frac{1}{p}+\frac{1}{p}\left(\frac{p}{p-1}\right)^{\xi} \alpha_k(p)^{-\xi} \right)      \nonumber \\
&= \prod_{p> \xi}\left(1-\frac{1}{p}+\frac{1}{p}\left(1+\frac{1}{p-1}\right)^{\xi} \left(1-\frac{(-1)^{\frac{k(p-1)}{4}}}{p^{\frac{k}{2}}}\right)^{-\xi} \right ) \nonumber \\
&\ll \prod_{p> \xi} \left(1+O\left(\xi p^{-2} \right) \right).
\end{align}
This infinite product is bounded. The result follows by combining \eqref{bound1} and \eqref{bound2}.
\end{proof}
Similar bounds can be obtained for $G_{k, \beta}(1-(1-\delta)z, z)$, and we omit the proof. 
\begin{lemma}\label{bound14}
For $\xi>0$ and $(s, z) \in  S_{1-\varepsilon}^+$, there exists a constant $\kappa$ such that 
\begin{align*}
 G_{k, \beta}(1-(1-\delta)z, z)\ll  \exp \left( \xi \left( 2 \log \log \xi+2\log \frac{\kappa}{e} \right)\right).
\end{align*}
\end{lemma}
In the half plane $\{z: \ 0<\Re(z)\leq \lambda \}$, $I_{k, \beta}(1-z+\delta z, z)$ and $G_{k, \beta}(1-z+\delta z, z)$ are bounded by a constant. 
\begin{lemma}
For any $\lambda>0$, on the half plane $\{z: \ 0<\Re(z)\leq \lambda \}$,
\[|  I_{k, \beta}(1-z+\delta z, z)|=O(1)\quad \text{and} \quad |  G_{k, \beta}(1-z+\delta z, z)|=O(1). \]
\end{lemma}
 \begin{proof}In here, we present a proof for the function $I_{k, \beta}(1-z+\delta z, z)$. The proof for $G_{k, \beta}(1-z+\delta z, z)$ is similar. From definition 
\begin{align*}
   |  I_{k, \beta}(1-z+\delta z, z)|&=  \left|\prod_{p}\left(1-\frac{1}{p}+\frac{1}{p}\left(\frac{p}{p-1}\right)^z \alpha_k(p)^{-z}\right) \right|\\ &=   \prod_{p}\left(1-\frac{1}{p}+\frac{1}{p}\left(\frac{p}{p-1}\right)^{\Re(z)} \alpha_{k}(p)^{-\Re(z)}\right)   \\
&\leq \prod_{p} \left(1-\frac{1}{p}+\frac{1}{p} \left(\frac{p}{p-1}\right)^{\lambda} \left(1+\frac{\lambda}{p^{k/2}-1}+O\left(\frac{\lambda^2}{(p^{k/2}-1)^2}\right)\right) \right)\\
     &\leq\prod_{p}\left(1-\frac{1}{p}+\frac{1}{p}\left(1+\frac{1}{p-1}\right)^\lambda \left(1+\frac{\lambda}{p^{k/2}-1}+O\left(\frac{\lambda^2}{(p^{k/2}-1)^2}\right) \right)\right)\\
     &\leq\prod_{p}\left(1-\frac{1}{p}+\frac{1}{p}\left( 1+\frac{2\lambda}{p-1}+O\left(\frac{\lambda^2}{(p-1)^2}\right) \right) \right)
    \\&\leq\prod_{p}\left(1+\frac{4\lambda}{p^2}+O_{\lambda}\left(\frac{1}{p(p-1)^2} \right)\right)
    =O(1).
    \end{align*} 
\end{proof}
\section{Key Lemmas} \label{key}
\begin{lemma}\label{M1}
Let $x \geq 1$ be any real number. Then
 \[\sum_{n\leq  x} \Phi_k(n)^{-z} n^{\beta z}=R_{k, \beta} (z) \frac{x^{1-z+\delta z}}{1-z+\delta z}+O\left(x^{1-b+\delta b}\exp{\left(- \sqrt{f_k \log x \log\log x}\right)}\right),\]
 where 
 \begin{align} \label{residue00}
  R_{k, \beta}(z)=\begin{cases}
  G_{k, \beta} (1-(1-\delta)z, z)& \quad \text{if } k \text{ is odd},\\
  I_{k, \beta} (1-(1-\delta)z, z)&  \quad \text{if } k \text{ is even}.
  \end{cases}
 \end{align}
for $x\geq x_0(k)$ and $z$ in the rectangle 
\begin{align}\label{region}
 \left\{z=b+i \eta :0< (1-\delta)b \leq \frac{1}{2},|\eta|\leq \exp \left(\left(\frac{1}{2}\log x \log\log x\right)^{\frac{1}{2}}\right)\right\}.
 \end{align}
\end{lemma}
\begin{proof}
We will prove it for $k$ even. The proof for odd $k$ follows similarly.
Employing Proposition \ref{Perron}, with $a=1+\frac{\log\log x}{\log x} $ and  $1\leq T\ll x$, we obtain
     \begin{align}\label{eqn:sum1}
         \sum_{n\leq  x} \Phi_k(n)^{-z}n^{\beta z}=\frac{1}{2\pi i}\int_{a-i T}^{a+i T}F_{k, \beta}(s,z)\frac{x^s}{s}\,dx +O\left( \frac{x^a}{T} \sum_{n=1}^{\infty}\frac{ \Phi_k(n)^{-b}}{n^{a- b\beta} |\log \frac{x}{n}|}\right).  
         \end{align}
         We first estimate the error term $R$ in \eqref{eqn:sum1}.
     \begin{align}\label{E11}
        R    &\ll      \frac{x^a}{T}  { \sum_{n=1}^{\infty}\frac{(\log \log n)^{2b}}{n^{a+b-b\delta} \log (\frac{x}{n})} }\\
             & \leq \frac{x^a}{T}  \left(\sum_{n\leq \frac{x}{e}} \frac{(\log \log n)^{2b}}{n^{a+b-b \delta} }+    \sum_{\frac{x}{e}<n<ex}\frac{(\log \log n)^{2b}}{n^{a+b-b \delta} \log (\frac{x}{n})}  +\sum_{n\geq ex} \frac{(\log \log n)^{2b}}{n^{a+b-b \delta}}\right)        \nonumber \\
             & \leq \frac{x^a(\log \log x)^{2b}}{T}  \left(\sum_{n\leq \frac{x}{e}} \frac{1}{n^{a+b-b \delta} }+    \left(\frac{e}{x}\right)^{a+(1- \delta)b}\sum_{\frac{x}{e}\leq n\leq ex}  \frac{ex}{|n-x|}  +\sum_{n\geq ex} \frac{1}{n^{a+b-b \delta}}\right) \nonumber \\
             &  \leq \frac{x^{1-(1-\delta)b}}{T}  (\log \log x)^{2b} \log x.
             \end{align} 
We now draw attention to the integral in \eqref{eqn:sum1}, which we denote by $I$. For this, we consider the contour consisting of the line segments $[a- i T, a+ i T], [a+ i T, c + i T], [c + i T, c- i T],\ \text{and}\ [c - iT, a-iT]$, where $c=1-(1- \delta)b-\sqrt{(\log \log x)/(2\log x)}$. Since the integrand has a simple pole inside the rectangle at $s=1-(1-\delta)z$, by Cauchy's residue theorem, this implies that
    \begin{align}\label{perron}
   I  
     =&R_{k, \beta}(z)\frac{x^{1-z+\delta z}}{1-z+\delta z}+ \frac{1}{2\pi i}\left\{\int_{a-i T}^{c-i T}+\int_{c-i T}^{c+i T}+\int_{c+i T}^{a+i T}\right\}F_{k, \beta}(s,z)\frac{x^s}{s} \ ds \nonumber \\
     =&R_{k, \beta}(z)\frac{x^{1-z+\delta z}}{1-z+\delta z}+J_1+J_2+J_3,
  \end{align}
  where  $R_{k, \beta}(z)$ is the residue. It is given by 
  \begin{align}\label{residue}
  R_{k, \beta}(z)=\begin{cases}
  G_{k, \beta} (1-(1-\delta)z, z)& \quad \text{if } k \text{ is odd},\\
  I_{k, \beta} (1-(1-\delta)z, z)&\quad \text{if } k \text{ is even}.
  \end{cases}
  \end{align}
  Next, we estimate the line integrals. We have
    \begin{align}\label{integral111}
  J_3     &\leq\frac{1}{T}\int_{c }^{a } x^\sigma|\zeta((1-\delta)z+\sigma+iT)||I_{k, \beta}(\sigma+iT, b+i\eta)| \ d\sigma \nonumber \\
        &\leq\frac{ \exp\left(\dfrac{M_k\log |\eta|}{\log\log |\eta|}\right) }{T}\int_{c }^{a } x^\sigma|\zeta(z(1-\delta)+\sigma+i T) |\ d\sigma \nonumber \\
        & \ll \frac{\exp\left(\frac{N_k\log x}{\log \log x}\right)^{1/2}}{T} \int_{c }^{a } x^\sigma|\zeta(z(1-\delta)+\sigma+i T) |\ d\sigma.
    \end{align}
    The second last step follows from Lemma \ref{bound12} since
    \[\sigma+b(1-\delta)\geq c+b(1-\delta)=1-\left(\frac{\log \log x}{2 \log x} \right)^{1/2}\geq 1-\frac{\log \log |\eta|}{\log |\eta|}.  \]
 Using bounds for $\zeta(s)$ \eqref{zetabounds} and above, the integral in \eqref{integral111} can be rewritten as
    \begin{align}\label{integral222}
  \int_{c }^{a } x^\sigma|\zeta(z(1-\delta)+\sigma+i T) |\ d\sigma  & \ll \left(\int_{c}^{1} +\int_{1}^{a}  \right)x^\sigma|\zeta((1-\delta)z+\sigma+i T) | \ d\sigma       \nonumber \\
    & \ll \left(\int_{c}^{1}  x^\sigma T^{\frac{1}{2}(1-\sigma-(1-\delta)b)}  \ d\sigma +\int_{1}^a  x^\sigma  \ d\sigma\right)        \nonumber \\
    & \ll T^{\frac{1}{2}(1-(1-\delta)b)}\int_{c}^{1}  \left(\frac{x}{\sqrt{T}}\right)^\sigma   \ d\sigma + \int_{1}^a  x^\sigma  \ d\sigma \nonumber \\
   & \ll x+x^{1-(1-\delta)b} T^{\frac{1}{2}{\left(\frac{\log \log x}{2\log x}\right)^{1/2}} }\exp(- (\frac{1}{2}\log x \log \log x)^{\frac{1}{2}}).
    \end{align}
   This leads to the estimate
    \begin{align}\label{integral333}
     J_3 & \ll  x \frac{\exp\left(\frac{N_k\log x}{\log \log x}\right)^{1/2}}{T}+ \frac{x^{1-(1-\delta)b}  \exp(-N_k^{\prime} (\log x \log \log x)^{\frac{1}{2}})}{T^{1-\frac{1}{2}{\left(\frac{\log \log x}{2\log x}\right)^{1/2}}}}. 
    \end{align}
     Similar estimates follow for $J_1$. Next, we focus on the integral $J_2$:
   \begin{align}\label{integral444}
  J_2 &\ll  x^c \int_{-T }^{T }  \frac{\left|\zeta((1-\delta)z+c+it)  I_{k, \beta}( c+it, z) \right|}{|c+it|} \ dt \nonumber\\
 & \ll \left(x^{(1-(1-\delta)b)} +x ^{1-(1-\delta)b}  T^{\frac{1}{2}{\left(\frac{\log \log x}{2\log x}\right)^{1/2}} }\right)\exp\left(-\frac{1}{\sqrt{2}} (\log x \log \log x)^{\frac{1}{2}}\right) \exp\left(\frac{N_k\log x}{\log_2 x}\right)^{1/2},
   \end{align} 
   where we have used bounds from Lemma \ref{bound12} and \eqref{zetabounds}.  The choice of \[T=x^{(1-\delta)b} \exp(\left(e_k\log x \log \log x\right)^{\frac{1}{2}}),\]
   leads to the following bounds for $J_1$ and $J_3$ in \eqref{integral333}:
    \begin{align}\label{integral3333}
 J_1, J_3    \ll x^{(1-(1-\delta)b)} \exp(-e_k^{\prime} (\log x \log \log x)^{\frac{1}{2}})
    \end{align}
    for $x \geq x(k)$. Substituting \eqref{residue}, \eqref{integral444} and \eqref{integral3333} in \eqref{perron}, we obtain
    \begin{align*}
     \frac{1}{2\pi i}\int_{a-i T}^{a+i T}F_{k, \beta}(s,z)\frac{x^s}{s}\ ds =R_{k, \beta}(z)\frac{x^{1-z+\delta z}}{1-z+\delta z}+O(x^{(1-(1-\delta)b)} \exp(-f_k (\log x \log \log x)^{\frac{1}{2}}))
     \end{align*} 
 which holds true in the rectangular region given in \eqref{region} $x\geq x(k)$. This, combined with \eqref{eqn:sum1} and \eqref{E11}, concludes the proof.
    \end{proof}
Next, we evaluate $\phi_{\Phi_k, \beta}(x, y)$. In doing so, we use Proposition \ref{Perron0} and choose $\tau=\exp \left(\frac{1}{2}\log x \log\log x\right)^{\frac{1}{2}}$. The constant $b$ will be chosen later depending on the values of $\delta$ and $\alpha$.

\begin{lemma}\label{Main0} For $x\geq x_{\delta, k}$,  there exist positive constants $f_k^{\prime}$ and $q_k$ such that         \begin{align*}
         \phi_{\Phi_k, \beta}(x, y)=&\frac{x}{2\pi i}\int_{b-i \tau}^{b+i \tau} R_{k, \beta}(z)\frac{(\alpha x^{\delta})^z }{z(1-(1-\delta)z)}\,dz+E_1,
              \end{align*}
              where $E_1$ is the error term given by
              \begin{align*}
          E_1\ll  \begin{cases}
           x^{1-\varepsilon+\delta \varepsilon} y^{\varepsilon}    \exp( -\sqrt{f_k^{\prime}\log x\log\log x})& \quad \text{if } 0\leq \delta \leq 1,\\
          x^{\frac{1}{2}+(1-\delta)\varepsilon} y^{\frac{1}{2(1-\delta)}-\varepsilon} \exp( -\sqrt{q_k\log x\log\log x})& \quad \text{if } \delta < 0 \ \mbox{and} \  0<\alpha <1,\\
           e^{\delta} x \alpha^{\frac{1}{\log x}}     (\log x \log\log x )^{\frac{1}{2}}&\quad  \text{if } \delta <0 \ \mbox{and} \  1\leq \alpha.
\end{cases}
              \end{align*}
        \end{lemma}
        
\begin{proof}
 Taking $\frac{\Phi_k(n)}{n^{\beta}}$ in Proposition \ref{Perron0} and $b>0$, we obtain
\begin{align}\label{M111}
 \phi_{\Phi_k, \beta}(x, y)= \frac{1}{2\pi i}\int_{b-i U}^{b+i U}\left(\sum_{n\le  x} \left(\frac{\Phi_k(n)}{n^\beta}\right)^{-z}\right)\frac{y^z}{z}\,dz+O\left( \frac{y^{b}}{U} \sum_{n \leq x} \Phi_k(n)^{-b} n^{b\beta}  \right).
  \end{align}
\begin{description}
 \item[Case (1)]If $0\leq \delta\leq 1$, then we can employ Lemma \ref{M1} (as $(1-\delta)\geq 0$ implies $b>0$) with  $0<b=\varepsilon$ and $U=\tau$ to evaluate the integral in the expression \eqref{M111}.
\begin{align*}
 \frac{1}{2\pi i}\int_{b-i \tau}^{b+i \tau}\left(\sum_{n\le  x} \left(\frac{\Phi_k(n)}{n^\beta}\right)^{-z}\right)\frac{y^z}{z}\,dz&= \frac{x}{2\pi i}\int_{b-i \tau}^{b+i \tau} R_{k, \beta} (z) \frac{(\alpha x^\delta) ^{z} }{z(1-z+\delta z)}dz\\
& +O\left(x^{1-\varepsilon+\delta \varepsilon} y^{\varepsilon}         \exp{(- \sqrt{f_k^\prime \log x \log\log x}} )\right).
\end{align*}
Again employing Lemma \ref{M1} the error term in \eqref{M111} can be estimated as $\ll  \frac{y^{b}}{\tau} x^{1-b+\delta b} $ and $0<\alpha < 1$ hence the result.
\item[Case (2)] If $\delta < 0$, then we have two cases: $0<\alpha < 1$ and $\alpha\geq 1$. 
\begin{description}
\item[Subcase (1)]If $0<\alpha < 1$, we will choose $b=\frac{1}{2(1-\delta)} -\varepsilon$ and  $U=\tau$ in Lemma \ref{M1} to evaluate the integral in the expression \eqref{M111} to obtain
\begin{align*}
 \frac{1}{2\pi i}\int_{b-i \tau}^{b+i \tau}\left(\sum_{n\le  x} \left(\frac{\Phi_k(n)}{n^\beta}\right)^{-z}\right)\frac{y^z}{z}\,dz&= \frac{x}{2\pi i}\int_{b-i \tau}^{b+i \tau} R_{k, \beta} (z)\frac{(\alpha x^\delta) ^{z} }{z(1-z+\delta z)}dz\\
& +O\left(x^{\frac{1}{2}+(1-\delta)\varepsilon} y^{\frac{1}{2(1-\delta)}-\varepsilon}\exp( -\sqrt{q_k\log x\log\log x})  \right).
\end{align*}
\item[Subcase (2)] For $\alpha\geq 1$,
we take $b=\frac{1}{\log x}$. It is easy to see that there exists $x_{\delta}\geq 0$ such that $2(1-\delta )<\log x$ for all $x\geq x_{\delta}$. We will apply Proposition \ref{Perron0} with  $U=\exp ((\log x \log\log x)^{\frac{1}{2}})$.  
Then \eqref{M111} can be evaluated as 
 \begin{align*}
  \phi_{\Phi_k, \beta}(x, y) &= \frac{1}{2\pi i}\left(\int_{b-i \tau}^{b+i \tau}   +   \int_{b-i U}^{b-i \tau} +\int_{b+i \tau}^{b+i U} \right) \left(\sum_{n\le  x} \left(\frac{\Phi_k(n)}{n^\beta}\right)^{-z}\right)\frac{y^z}{z}\,dz\\
  &+O\left( \frac{y^{b}}{U} \sum_{n \leq x} \Phi_k(n)^{-b} n^{b\beta}  \right).
 \end{align*}
 Now the error term can be $\ll  \frac{\alpha^{b}}{U} x^{1+\delta b} $.
For the integral in the range  $[b-i \tau, b+i \tau]$ we will use Lemma \eqref{M1} to obtain
 \begin{align*}
 \frac{1}{2\pi i}\int_{b-i \tau}^{b+i \tau}\left(\sum_{n\le  x} \left(\frac{\Phi_k(n)}{n^\beta}\right)^{-z}\right)\frac{y^z}{z}\,dz&= \frac{x}{2\pi i}\int_{b-i \tau}^{b+i \tau} R_{k, \beta} (z) \frac{(\alpha x^\delta) ^{z} }{z(1-z+\delta z)}dz\\
& +O\left(x^{1+\delta b} \alpha^{b}         \exp{(- \sqrt{f_k^{\prime \prime \prime} \log x \log\log x}} )\right).
\end{align*}
  We estimate $ \int_{b-i U}^{b-i\tau} $ and $\int_{b+i \tau}^{b+i U} $ by observing $ |\sum_{n \leq x} \Phi_k(n)^{-b} n^{b\beta}|\ll x^{1+\delta b}$ and $|y^z|\leq |(\alpha x)^z|\leq e\alpha^b$. Hence
  \begin{align*}
\left(  \int_{b-i U}^{b-i\tau} +\int_{b+i \tau}^{b+i U}   \right) \left(\sum_{n\le  x} \left(\frac{\Phi_k(n)}{n^\beta}\right)^{-z}\right)\frac{y^z}{z}\,dz\ll e\alpha^b x^{1+\delta b}  (\log x \log\log x )^{\frac{1}{2}}.
\end{align*}
 \end{description}
  \end{description}
\end{proof}
\begin{lemma}\label{Main00}
There exists an absolute positive constant $\kappa$ such that for any $\delta\leq 0$ with any $y\geq1$  and for $0< \delta <1$ with $0< y\leq \frac{(\log ((1-\delta)^{-1}+\varepsilon))^{-4}}{c_k \kappa^4}x^{1-\delta}$ we have
\begin{align}\label{Pse}
\frac{x}{2\pi i}\int_{b-i \tau}^{b+i \tau}R_{k, \beta} (z) \frac{(\alpha x^\delta) ^{z} }{z(1-z+\delta z)}dz=R_{k, \beta} \left(\frac{1}{(1-\delta)}\right)y^{\frac{1}{(1-\delta)}} +E_2,
\end{align}
where $E_2$ is the error term given by \begin{align*}
E_2 \ll_{\delta} \begin{cases}
x\exp \left( -4 \exp\left( \kappa (c_k \alpha x^\delta)^{\frac{1}{4}}\right)^{-1}\right)&  \quad \text{if} \ k \ \text{is even}, \\
  x \exp \left( -2 \exp\left( \kappa ( \alpha x^\delta)^{\frac{1}{2}}\right)^{-1}\right)& \quad \text{if} \ k \ \text{is odd}.
\end{cases}
\end{align*}

for  $x\geq x_{{\beta}, k}$.
\end{lemma}
\begin{proof}
We will consider the case when $k$ is even. The case of $k$ being odd is similar.
We will denote the above integral by $I_{(b, \tau)}$. We take the contour consisting of the line segments $[b- i \tau, b+ i \tau], [b+ i \tau, d + i \tau], [d + i \tau, d- i \tau], [d - i\tau, b-i\tau]$, where $d>\frac{1}{1-\delta} $ is a parameter to be chosen later. Since integrand has a simple pole inside the rectangle at $z=\frac{1}{(1-\delta)}$ so by Cauchy's residue theorem, we have
\begin{align*}
I_{(b, \tau)}=R_{k, \beta} \left(\frac{1}{(1-\delta)}\right)y^{\frac{1}{(1-\delta)}} +\frac{x}{2\pi i}\left( \int_{d-i \tau}^{d+i \tau} +\int_{d+i \tau}^{b+i \tau} +\int_{b-i \tau}^{d-i \tau}\right) R_{k, \beta} (z) \frac{(\alpha x^\delta) ^{z} }{z(1-z+\delta z)}dz.
\end{align*}
To estimate the horizontal integrals, we use Lemma \ref{bound13} and will get 
\begin{align*}
&\int_{d+i \tau}^{b+i \tau} R_{k, \beta} (z) \frac{(\alpha x^\delta) ^{z} }{z(1-z+\delta z)}dz\\
&\ll_{\delta}  \exp \left( \xi \left( \log c_k+ 4 \log \log \xi+4\log \frac{\kappa}{e} \right)\right) \tau^{-2} \int_{b}^{d} (\alpha x^\delta) ^{u} du\\
&\ll_{\delta} \exp( -(2\log x\log\log x)^{\frac{1}{2}})  \exp \left( d\left( \log c_k+ 4 \log \log d+4\log \frac{\kappa}{e} \right)\right)   (\alpha x^\delta)^{d}. 
\end{align*}
Similarly, the other horizontal integral can be treated. It remains to bound the integral $\int_{d-i \tau}^{d+i \tau} $. 
\begin{align*}
\int_{d-i \tau}^{d+i \tau} R_{k, \beta} (z) \frac{(\alpha x^\delta) ^{z} }{z(1-z+\delta z)}dz\ll_{\delta} \exp \left( d\left( \log c_k+ 4 \log \log d+4\log \frac{\kappa}{e} \right)\right)  (\alpha x^\delta)^{d} \left(1+\frac{1}{\tau}\right).
\end{align*}
Combining these estimates, we obtain
\begin{align}\label{ek}
I_{(b, \tau)}&=R_{k, \beta} \left(\frac{1}{(1-\delta)}\right)y^{\frac{1}{(1-\delta)}} +O_{\delta} \left(x\exp \left( d\left( \log c_k+ 4 \log \log d+4\log \frac{\kappa}{e} \right)\right)  (\alpha x^\delta)^{d}\right) \nonumber\\
&=R_{k, \beta} \left(\frac{1}{(1-\delta)}\right)y^{\frac{1}{(1-\delta)}} +O_{\delta} \left(x\exp \left( d\left( \log c_k+ 4 \log \log d+4\log \frac{\kappa}{e} +\log(\alpha x^\delta) \right)\right)  \right).
\end{align}
We will choose $d$ in a way such that $\left( \log c_k+ 4 \log \log d+4\log \frac{\kappa}{e} +\log(\alpha x^\delta) \right)$ becomes negative and less than equal to $-4$ which happens only when $d\leq \exp\left( (\kappa (c_k \alpha x^\delta)^{\frac{1}{4}})^{-1}\right)$.
We take $d= \exp\left( (\kappa (c_k \alpha x^\delta)^{\frac{1}{4}}    )^{-1}\right)$. Then the expression in \eqref{ek} reduces to
\begin{align}\label{ej}
I_{(b, \tau)}
&=R_{k, \beta} \left(\frac{1}{(1-\delta)}\right)y^{\frac{1}{(1-\delta)}} +O_{\delta} \left(x \exp \left( -4 \exp\left( \kappa (c_k \alpha x^\delta)^{\frac{1}{4}}\right)^{-1}\right)  \right).
\end{align}
The case for odd $k$ will be similar except in that case; we will use Lemma \ref{bound14} and choose $d= \exp\left( (\kappa ( \alpha x^\delta)^{\frac{1}{2}}    )^{-1}\right)$. The error term in that case will be $\ll_{\delta} x\exp \left( -2 \exp\left( \kappa ( \alpha x^\delta)^{\frac{1}{2}}\right)^{-1}\right)$.

For the case of $\delta \leq 0$, the choice of $d$ automatically satisfies the criterion $d>\frac{1}{(1-\delta)}$.
But in case of $0<\delta<1$, we have $d= \exp\left( ( \kappa (c_k \alpha x^\delta)^{\frac{1}{4}})^{-1} \right) \geq \frac{1}{(1-\delta)}+\varepsilon $ for any arbitarily small $\varepsilon > 0$ iff $y \leq  \frac{(\log ((1-\delta)^{-1}+\varepsilon))^{-4}}{c_k \kappa^4}x^{1-\delta}$.
\end{proof}
This leads finally to the case $0< \delta <1$ for $\frac{(\log ((1-\delta)^{-1}+\varepsilon))^{-4}}{c_k \kappa^4}x^{1-\delta} <y \leq x$, summarized below.
\begin{lemma}\label{Main000}
For $0< \delta <1$ with $ \frac{(\log ((1-\delta)^{-1}+\varepsilon))^{-4}}{c_k \kappa^4}x^{1-\delta} <y \leq x$, we have
\begin{align*}
\frac{x}{2\pi i}\int_{b-i \tau}^{b+i \tau}R_{k, \beta} (z) \frac{(\alpha x^\delta) ^{z} }{z(1-z+\delta z)}dz\ll_{\delta} x^{1-\varepsilon+\delta \varepsilon} y^{\varepsilon}    \exp\left( -\sqrt{f_k^{\prime}\log x\log\log x}\ \right),  
\end{align*}
for  $x\geq x_{{\beta}, k}$.
\end{lemma}
\section{Proofs of Main Theorems} \label{Mainthm}
\begin{proof}[Proof of Theorem \ref{thm3}]
When $0<\delta< 1$ and $0<\alpha<1$, we note that the term $E_1$ defined in Lemma \ref{Main0} will  dominate the error term $E_2$ defined in the previous Lemma \ref{Main00}. But for the main term in \eqref{Pse} to dominate $E_1$, we need the condition stated in Theorem \ref{thm3}. We get the desired result by combining Lemma \ref{Main0} and Lemma \ref{Main00}.
\end{proof}
Furthermore,

\begin{proof}[Proof of Theorem \ref{thm4}]
The expression in Lemma \ref{Main0} can be estimated as 
\begin{align*}
\phi_{\Phi_k, \beta}(x, y) \ll_{\delta} x^{1-(1-\delta) \varepsilon}  y^{\varepsilon}  \exp\left( -\sqrt{p_k^{\prime}\log x\log\log x}\ \right), 
\end{align*}
and hence the result.\end{proof}
\begin{proof}[Proof of Theorem \ref{thm5}]
The  integral in Lemma \ref{Main0} results into
\begin{align*}
\phi_{\Phi_k, \beta}(x, y) \ll_{\delta}x^{\frac{1}{2}+(1-\delta)\varepsilon} y^{\frac{1}{2(1-\delta)}-\varepsilon} \exp\left( -\sqrt{q_k\log x\log\log x}\ \right), 
\end{align*}
and hence the result.
\end{proof}
\begin{proof}[Proof of Theorem \ref{thm6}]
Combining Lemma \ref{Main0} and Lemma \ref{Main00}, we obtain
\begin{align}\label{eqthm6}
\phi_{\Phi_k, k-1}(x, y)&= R_{k, k-1} (1)y  +O(x  \exp\left( -\sqrt{f_k^{\prime }\log x\log\log x}\ \right).     
\end{align}
The proof is complete if we express the error term in \eqref{eqthm6} in terms of $y$, for which we follow the same reasoning as in \cite{D}.
\end{proof}
\begin{proof}[Proof of Corollary \ref{cor1}]
For $k=1$ and $\beta=0$ or $\delta=0$, the constant $R_{1,0}(1)$ can be equated to $\frac{\zeta(2)\zeta(3)}{\zeta(6)}$.
\end{proof}
\begin{proof}[Proof of Theorem  \ref{thm7}]
We first remark that there exists a $x^{\prime}_{\beta, k}$ such that $x \geq (1-\delta)$ for all $x \geq x^{\prime}_{\beta, k}$ and $\delta  <1$.

Note that for $\delta< 0$ and $\alpha\geq 1$, the term $E_1$ defined in Lemma \ref{Main0} dominates the error term $E_2$ defined in Lemma \ref{Main00}. 
Moreover, in this case, if $x (\log x \log\log x )^{\frac{1}{2}} < e^{1-\delta}y^{\frac{1}{1-\delta}-\frac{1}{\log x}} $, then the main term in Lemma \ref{Main00} is bigger than the error term $E_1$. Otherwise, we only obtain bounds for $\phi_{\Phi_k, \beta}(x, y)$.
Hence combining Lemma \ref{Main0} and Lemma \ref{Main00}, we obtain
\begin{align*}
\phi_{\Phi_k, \beta}(x, y)&= R_{(k, \beta)} \left(\frac{1}{(1-\delta)}\right)y^{\frac{1}{(1-\delta)}}   +O_\delta( x \alpha^{\frac{1}{\log x}}     (\log x \log\log x )^{\frac{1}{2}} ).     
\end{align*}
And in the other case, the integral in Lemma \ref{Main0} can be estimated as 
\begin{align*}
\frac{x}{2\pi i}\int_{(b, \tau)} R_{k, \beta}(z)\frac{(\alpha x^{\delta})^z }{z(1-(1-\delta)z)}\,dz \ll_{\delta} x \alpha^{\frac{1}{\log x}}     (\log x \log\log x )^{\frac{1}{2}}.
\end{align*}
and hence the result
\[\phi_{\Phi_k, \  \beta}(x, y)      \ll_{\delta} \min(x, 
      ( x \alpha^{\frac{1}{\log x}}     (\log x \log\log x )^{\frac{1}{2}} )) \ll_{\delta} x.  \] 

\end{proof}
\section{Maximal and Minimal order of $\Phi_k(n)$} \label{extremal}
This section is devoted to the proofs of Theorem \ref{kodd} and Theorem \ref{keven}.
\begin{proof}[Proof of Theorem  \ref{kodd}]
For $k$ odd, the proof follows directly from \cite{T}. 
If $k$ is even, we observe that $ \dfrac{\Phi_k(n)}{n^k}\leq1$.
Now, for a sequence $ \{p_j\}$ with $p_j$ primes and $p_j<p_{j+1}$ for $j\geq 1$, $\lim_{j\rightarrow\infty} \frac{\Phi_k(p_j)}{p^k_j}=1$.
Hence the result.
\end{proof}
\begin{proof}[Proof of Theorem  \ref{keven}]
When $k$ is odd, the proof follows directly from \cite{T}. For $k$ even, we split it into two cases:
\begin{description}
    \item [Case (1)]
Let $k\equiv 0\pmod{4}$. To evaluate the minimal order, we observe that
\begin{align*}
\Phi_k(n) \geq &    \left( \frac{2^{\frac{k}{2}}     }{2^{\frac{k}{2}}-1}\right) n^{k-1} \phi(n)\prod_{p \leq n} \left(1-\frac{1}{p^{\frac{k}{2}}}\right) \\
 \Phi_k(n) \geq & \left( \frac{2^{\frac{k}{2}}     }{2^{\frac{k}{2}}-1}\right)\left( \zeta\left(\frac{k}{2}\right)^{-1} +O\left(\frac{1}{n^{\frac{k}{2}-1}} \right)\right) \frac{e^{-\gamma} n^k}{\log\log n}.
\end{align*} 
In the last step, we used Proposition \ref{mertengeneral}.
Let us construct the sequence $\{n_s\}$ with $n_s=\prod_{1\leq i \leq s }p_i$ for $s\geq 1$. Then, 
\begin{align*}
\log n_s= \sum_{p \leq p_s}\log p=p_s\left(1+O\left(\frac{1}{ p_s}\right)\right).
\end{align*}
and  using this, we obtain
\begin{align*}
\Phi_k(n_s)&=  \left( \frac{2^{\frac{k}{2}}     }{2^{\frac{k}{2}}-1}\right)n_s^{k-1}\phi(n_s)\prod_{p \leq n_s} \left(1-\frac{1}{p^{\frac{k}{2}}}\right)\\
&= \left( \frac{2^{\frac{k}{2}}     }{2^{\frac{k}{2}}-1}\right)\left( \zeta\left(\frac{k}{2}\right)^{-1} +O\left(\frac{1}{n_s^{\frac{k}{2}-1}} \right)\right) \frac{e^{-\gamma} n_s^k}{\log\log n_s}.
\end{align*}
\item [Case (2)] If $k\equiv 2\pmod{4}$, we observe that
\begin{align*}
 \Phi_k(n)&\geq  n^{k} \prod_{p \leq n}\left(1-\frac{1}{p}\right) \left(1-\frac{\chi_1(p)}{p^{\frac{k}{2}}}\right)\\
&\geq   \begin{cases}
        \left(L\left(1, \chi_1\right)^{-1}+ O(\frac{1}{\log x})\right)\frac{e^{-\gamma} n^k}{\log\log n}& \quad \text{if } k=2,\\
        \left( L\left(\frac{k}{2}, \chi_1\right)^{-1} +O\left(\frac{1}{n^{\frac{k}{2}-1}} \right)\right) \frac{e^{-\gamma} n^k}{\log\log n} & \quad \text{if } k\neq 2.
\end{cases}
              \end{align*}
Here we have used Proposition \ref{mertengeneralap}.
Then, considering the sequence $\{n_t\}$ with $n_t=\prod_{1\leq i \leq t }p_i$
and proceeding as above, we obtain the required result.
\end{description}

\end{proof}

{\bf Conflict of Interest} The authors declare that they have no conflict of interest.

\end{document}